\documentclass[graybox]{svmult}

\usepackage{bookmark}
\makeatletter
\def\toclevel@title{-1}
\def\toclevel@author{0}
\makeatother

\usepackage{type1cm}        %
\usepackage{makeidx}         %
\usepackage{graphicx}        %
\usepackage{multicol}        %
\usepackage[bottom]{footmisc}%

\makeindex             %

\usepackage{pgfplots}
\pgfplotsset{compat=newest}

\usepackage{amssymb} %
\usepackage{amsmath}
\usepackage{mathtools}
\usepackage[textsize=tiny]{todonotes}
\usepackage[printonlyused]{acronym}
\usepackage{hyperref}
\usepackage[capitalise]{cleveref}
\DeclareMathSymbol{\shortminus}{\mathbin}{AMSa}{"39}

\usepackage{newtxtext}       %
\usepackage[varvw]{newtxmath}       %

\allowdisplaybreaks

\usepackage{background}

\backgroundsetup{
	pages=all,
	color=black,
	placement=top,
	scale=1.0,
	opacity=1,
	vshift=-0.5cm,
	contents={
		\begin{minipage}{1.5\linewidth}
			\scriptsize
			This is a preprint of the following chapter: Jan-Christopher Cohrs and Benjamin Berkels, \emph{On the importance of the \texorpdfstring{$\varepsilon$}{epsilon}-regularization of the distribution-dependent Mumford--Shah model for hyperspectral image segmentation}, published in \emph{Multiscale, Nonlinear and Adaptive Approximation II}, edited by Ronald DeVore and Angela Kunoth, 2024, Springer.
            It is the version of the authors' manuscript prior to acceptance for publication and has not undergone editorial and/or peer review on behalf of the Publisher (where applicable).
            The final authenticated version is available online at: \url{https://doi.org/10.1007/978-3-031-75802-7_10}.
		\end{minipage}
	}
}

\DeclareMathOperator{\divergence}{div}
\DeclareMathOperator{\per}{Per}

\newcommand{\AAA}{\mathcal{A}}

\newcommand{\dx}[1][x]{\;\mathrm{d}#1}

\newcommand{\gammaconvergence}{$\Gamma$-convergence}
\newcommand{\gammaconverges}{$\Gamma$-converges}

\newcommand{\gammalimit}{$\Gamma$-limit}
\newcommand{\imdom}{\Omega}
\newcommand{\indexset}[1][n]{\{1,\dots, #1\}}
\newcommand{\bv}{\text{BV}(\imdom)}
\newcommand{\lebesgue}[1]{\left\vert #1 \right\vert}
\newcommand{\N}{\mathbb{N}}
\newcommand{\OO}{\mathcal{O}}
\newcommand{\Pe}{\mathcal{P}_{\varepsilon}}
\newcommand{\perIm}[1]{\per\left(#1, \imdom\right)}
\newcommand{\Pp}{\mathcal{P}}
\newcommand{\R}{\mathbb{R}}
\newcommand{\RExt}{\overline{\mathbb{R}}}

\newcommand{\toweakstar}{\stackrel{*}{\rightharpoonup}}
\newcommand{\tv}[1]{\vert #1 \vert_{\text{BV}(\imdom)}}

\usepackage{bm}
\bmdefine{\bchi}{\chi}
\bmdefine{\bmu}{\mu}
\bmdefine{\bOO}{\OO}
\bmdefine{\bSigma}{\Sigma}
\bmdefine{\bsigma}{\sigma}
\bmdefine{\bV}{V}

\acrodef{ms}[MS]{Mumford--Shah}
\acrodef{bv}[BV]{bounded variation}
\acrodef{hsi}[HSI]{hyperspectral imaging}
\acrodef{pca}[PCA]{principal components analysis}
\acrodef{tv}[TV]{total variation}

\begin{document}

\title*{On the importance of the \texorpdfstring{$\varepsilon$}{epsilon}-regularization of the distribution-dependent Mumford--Shah model for hyperspectral image segmentation}

\titlerunning{Importance of \texorpdfstring{$\varepsilon$}{epsilon}-regularization of the distribution-dependent MS model}

\author{Jan-Christopher Cohrs\orcidID{0000-0001-9173-4015} and\\
Benjamin Berkels\orcidID{0000-0002-6969-187X}}

\authorrunning{J.-C. Cohrs, B. Berkels}

\institute{Jan-Christopher Cohrs \at Aachen Institute for Advanced Study in Computational Engineering Science (AICES), RWTH Aachen University, Germany, \email{cohrs@aices.rwth-aachen.de}
\and Benjamin Berkels \at Aachen Institute for Advanced Study in Computational Engineering Science (AICES), RWTH Aachen University, Germany, \email{berkels@aices.rwth-aachen.de}
}

\maketitle
\abstract*{Recently, the distribution-dependent Mumford--Shah model for hyperspectral image segmentation was introduced.
It approximates an image based on first and second order statistics using a data term, that is built of a Mahalanobis distance plus a covariance regularization, and the total variation as spatial regularization.
Moreover, to achieve feasibility, the appearing matrices are restricted to symmetric positive definite ones with eigenvalues exceeding a certain threshold.
This threshold is chosen in advance as a data-independent parameter.
In this article, we study theoretical properties of the model.
In particular, we prove the existence of minimizers of the functional and show its \gammaconvergence{} when the threshold regularizing the eigenvalues of the matrices tends to zero.
It turns out that in the \gammalimit{} we lose the guaranteed existence of minimizers; and we give an example of an image where the \gammalimit{} indeed has no minimizer.
Finally, we derive a formula for the minimum eigenvalues of the covariance matrices appearing in the functional that hints under which conditions the functional is able to handle the data without regularizing the eigenvalues.
The results of this article demonstrate the significance and importance of the eigenvalue regularization to the model and that it cannot be dropped without substantial modifications.}

\abstract{Recently, the distribution-dependent Mumford--Shah model for hyperspectral image segmentation was introduced.
It approximates an image based on first and second order statistics using a data term, that is built of a Mahalanobis distance plus a covariance regularization, and the total variation as spatial regularization.
Moreover, to achieve feasibility, the appearing matrices are restricted to symmetric positive definite ones with eigenvalues exceeding a certain threshold.
This threshold is chosen in advance as a data-independent parameter.
In this article, we study theoretical properties of the model.
In particular, we prove the existence of minimizers of the functional and show its \gammaconvergence{} when the threshold regularizing the eigenvalues of the matrices tends to zero.
It turns out that in the \gammalimit{} we lose the guaranteed existence of minimizers; and we give an example of an image where the \gammalimit{} indeed has no minimizer.
Finally, we derive a formula for the minimum eigenvalues of the covariance matrices appearing in the functional that hints under which conditions the functional is able to handle the data without regularizing the eigenvalues.
The results of this article demonstrate the significance and importance of the eigenvalue regularization to the model and that it cannot be dropped without substantial modifications.}

\section{Introduction}
Image processing comprises several problems, such as image registration, deblurring, denoising, inpainting or segmentation.
Image registration aims to align two or more images of the same object.
In image deblurring, one tries to recover the original image from a blurry version of it.
Denoising deals with the reduction of noise in an image.
The task of filling missing or damaged parts of an image is referred to as inpainting.
Image segmentation is concerned with the task to partition an image into meaningful regions where the term \emph{meaningful} depends on the respective setting, i.e., the image modality and which information one wants to extract from the image.
A well-established model in image processing is the \ac{ms} functional \cite{MuSh89}.
It is very popular in image segmentation, but also used for image restoration \cite{BaChCh11,JuBrCh11}.
For example, it has been adapted for inpainting \cite{EsSh02,HaRoKu22}, but variants are also applied in deblurring \cite{HaDuVo07}.
The functional in its original form aims to find non-trivial piecewise smooth approximations to images and consists of the squared difference as a data term plus two regularizing terms \cite{MuSh89}.
For a fixed set of segment boundaries, it is shown in \cite[Prop. 3]{Da05} with a convexity argument that a unique minimizer exists.
In \cite[Th. 1.1]{DeCaLe89} and for a slightly different version of the functional in \cite[Th. 2]{BaChCh11}, the authors prove that when the set of segment boundaries is replaced by the set of discontinuities of the approximating function, a minimizer of the functional exists.
Rondi and Santosa \cite{RoSa01} leverage the \ac{ms} functional for electrical impedance tomography.
They equip it with a data term measuring the difference of an operator applied to the approximating function and the original data in the norm of the space of bounded, linear operators that map from the space of traces of $L^2(\imdom)$ on $\partial \imdom$ with zero means to itself.
Moreover, they prove existence results for a \ac{ms} variant on two function classes for general data terms that satisfy a continuity condition.
Fornasier et al. propose a structurally similar \ac{ms}-based functional; however, their data term is the $L^2$-distance between a linear, non-invertible operator applied to the original image (in contrast to before where the operator was applied to the approximation) and the approximating function \cite{FoMaSo11}.
They provide counterexamples in which existence of minimizers cannot be guaranteed but prove two existence results for certain classes of linear operators.
The authors of \cite{CaChZe13} introduce a variant of the \ac{ms} model with the squared difference as data term including a blurring operator, but without the boundary regularization and prove the existence of a unique minimizer.
In \cite{CaChNi17}, this model is further modified with a data term that deals with images that are corrupted by Poisson noise; and also for this model the existence of a unique minimizer is proven.
Mumford and Shah further propose to restrict their functional to piecewise constant functions and show for that restricted model the existence of minimizers for continuous input images \cite[Th. 5.1]{MuSh89}.

Moreover, alternative data terms for the \ac{ms} functional have been introduced, where existence results have not been proven yet.
In \cite{TsYeWi01}, the authors generalize the original model by multiplying the data term with a location-dependent weighting factor to give less or no weight to noisy or corrupt pixels.
Ramlau and Ring \cite{RaRi07} build on the piecewise constant \ac{ms} model to introduce a data term for X-ray tomography, which they define as the squared $L^2$-distance of the data and the Radon transform of an unknown density function.

In this article, we will consider the distribution-dependent \ac{ms} model for hyperspectral image segmentation that was introduced in \cite{CoBaBe22}.
We will call it $\varepsilon$AMS in the following, short for \emph{$\varepsilon$-regularized anisotropic \ac{ms} model}.
It consists of the piecewise constant \ac{ms} functional where the squared $L^2$-distance in the data term was replaced by a non-squared Mahalanobis distance \cite{Ma36} plus a regularizing term to penalize covariance matrix estimates with large eigenvalues.
Covariance matrices are symmetric and positive semi-definite by definition.
However, $\varepsilon$AMS even requires them to be positive definite to ensure the invertibility and therefore the feasibility of the model.
To this end, a regularizing parameter $\varepsilon > 0$ is introduced that defines a lower bound for the eigenvalues of the covariance matrices and therefore keeps them positive by setting eigenvalues of covariance matrices to $\varepsilon$, if they are smaller than this threshold.
In terms of images one can think of $\varepsilon$ as the minimum variance of the spectra possible in a segment.
Consequences of this regularization are: it ensures invertibility of the covariance matrices and improves their conditions, the covariance matrices are optimized over a closed set and $\varepsilon$ ensures a lower bound of the functional.
However, $\varepsilon$ is a data-independent parameter that has to be chosen before the application of the model. This choice is not trivial and it is also hard to interpret why a specific choice leads to good results and others do not.
In particular, the choice of $\varepsilon$ potentially excludes minima by restricting the admissible set.
Nevertheless, the parameter $\varepsilon$ plays a crucial role in the considered model and we will see more arguments why we need it in $\varepsilon$AMS in the next section.

Since segmenting an image means finding minimizers of $\varepsilon$AMS, the question arises if minimizers exist.
We will state and prove in \cref{subsec:epsAMS-existence-minimizers} an existence result, guaranteeing the existence of minimizers of $\varepsilon$AMS.
Also in the proof, a key factor is the $\varepsilon$-regularization, which emphasizes its importance.
As a lower bound for the eigenvalues of the covariances matrices, $\varepsilon$ is supposed to be chosen small.
In \cref{subsec:epsAMS-gamma-convergence}, we will show the \gammaconvergence{} of the functional to a \gammalimit{} when $\varepsilon \to 0$.
It turns out that minimizers of the \gammalimit{} do not necessarily exist. In particular, we will see a counterexample for which the \gammalimit{} has no minimizer in \cref{subsec:epsAMS-unboundedness-J-0}.
Finally, we derive a formula depending on the data that shows the minimum values for the eigenvalues of the covariance matrices in \cref{subsec:epsAMS-boundedness-eigenvalues}, which helps to understand what goes wrong in case of the \gammalimit{}.

\section{The distribution-dependent Mumford--Shah model}
The model that we will study is a variant of the \ac{ms} segmentation functional \cite{MuSh89}, equipped with the indicator function proposed in \cite{CoBaBe22}.
We will see a detailed description of the model in \cref{subsec:epsAMS-model-description} and show the existence of minimizers of the model afterwards (cf. \cref{subsec:epsAMS-existence-minimizers}).
The model contains a (small) regularization parameter to ensure the feasibility of the model and which guarantees the existence of minimizers of the functional.
We will show the \gammaconvergence{} of the functional to a \gammalimit{} when this regularization parameter tends to $0$ in \cref{subsec:epsAMS-gamma-convergence}.
Under suitable conditions, not only the functional \gammaconverges{} to a \gammalimit{} but also the minimizers, which would give us the existence of minimizers also for the limit functional.
However, as we will see in \cref{subsec:epsAMS-unboundedness-J-0}, we do not have a guaranteed minimizer for the limit functional anymore, which emphasizes the importance of a regularization of the eigenvalues of the covariance estimates to bound them away from $0$.

\subsection{Model description of \texorpdfstring{$\varepsilon$}{epsilon}AMS}\label{subsec:epsAMS-model-description}
Throughout this article, let $d\in \N$ and $\imdom \subseteq \R^d$ be a measurable, bounded domain with Lipschitz boundary.
Furthermore, let $g \in L^\infty(\imdom, \R^L)$ be a hyperspectral image with $L \in \N$ channels, $\varepsilon, \eta > 0$ and $k\in \N$ the number of segments, which is chosen in advance.
Since $\imdom$ is bounded, $g\in L^p(\imdom, \R^L)$ for every $p \in [1,\infty)$.
For $A\subseteq\mathbb{R}^d$ measurable, let $\lebesgue{A}$ be the $d$-dimensional Lebesgue measure of $A$.
Note that for such sets $A$ we often equivalently consider the characteristic functions $\chi_A \colon \R^d \to \{0,1\}$ and switch between these two representations without further notice.
Let $\OO \subseteq \R^d$ be measurable.
Since $\imdom$ is bounded, we get $\lebesgue{\OO \cap \imdom} < \infty$ and thus $\chi_{\OO} \in L^1(\imdom)$.
Following \cite[Def. 3.4, Def. 3.35]{AmFuPa00}, we define the \emph{perimeter $\perIm{\OO}$ of $\OO$ in $\imdom$} as
\begin{equation*}
    \perIm{\OO} := \sup \left\{\int_\OO \divergence \varphi \dx \: \middle| \; \varphi \in C_c^1(\Omega, \R^d), \Vert \Vert \varphi \Vert_2 \Vert_{L^\infty(\imdom)} \leq 1 \right\}.
\end{equation*}
Because $\chi_\OO \in L^1(\imdom)$, we have $\perIm{\OO} < \infty$ $\Leftrightarrow$ $\chi_\OO \in \bv$, cf. \cite[Prop. 3.6]{AmFuPa00}. In particular, $\perIm{\OO}= \tv{\chi_\OO}$. 
Here, $\tv{\cdot}$ denotes the Total Variation semi-norm and $\bv$ the space of functions of Bounded Variation. For a general introduction to $\bv$ and sets of finite perimeter, we refer to \cite{AmFuPa00}.

The set of all partitions of $\imdom$ into $k$ measurable subsets with finite perimeter %
is
\begin{equation*}
    \begin{split}
    \AAA := 
    &\left\{ 
        \left( \OO_{1}, \OO_{2}, \dots, \OO_{k}\right) \; \middle| \;
        \forall l\in \indexset[k]: \OO_{l} \subseteq \imdom \text{ measurable},
        \vphantom{{\textstyle\bigcup\nolimits_{l=1}^k}}\right.\\
        &\left. \hphantom{\left\{\right.} \perIm{\OO_{l}} < \infty, \ \forall m \neq l: \OO_{l} \cap \OO_{m} = \emptyset, \ {\textstyle\bigcup\nolimits_{l=1}^k} \OO_{l} = \imdom
    \right\}.
    \end{split}
\end{equation*}
We want to point out here that we do not require the closure of the union to be equal to the closure of $\imdom$ but a slightly stronger condition.
The reason is that we prove the partition property later using the characteristic functions of the involved sets which have to sum up to one at every element of $\imdom$.
This would not be the case for the condition with the closures.
Then, the \ac{ms} segmentation functional \cite{MuSh89} is
\begin{equation}\label{eq:model-description-ms-functional}
        J_{\text{MS}} \colon \AAA \to \R,
        \left( \OO_{1}, \OO_{2}, \dots, \OO_{k}\right) \mapsto \sum_{l=1}^{k} \left[\int_{\OO_l} f_l(x) \dx + \lambda \perIm{\OO_l}\right],
\end{equation}
where $\lambda > 0$ is a parameter that balances the data term, i.e. the sum of the integrals over $\OO_l$, and the regularizing term, i.e. the sum over $\perIm{\OO_l}$.
The so-called \emph{indicator functions} $f_l \colon \imdom \to \R$ play a key role in the \ac{ms} functional.
They define the notion of homogeneity for segmentation
and have to be chosen according to the task at hand.
Let
\begin{equation}\label{eq:def-eigenvalue-operator}
    \sigma \colon \left\{M \in \R^{L \times L} \; \middle| \; M^T = M\right\} \to \R^L
\end{equation}
be the mapping that returns the eigenvalues of a symmetric matrix in decreasing order, i.e., $\sigma(M)_1 \geq \sigma(M)_2 \geq \dots \geq \sigma(M)_L$.
Furthermore, we define the set of all symmetric and positive definite matrices as
\begin{equation*}
    \Pp := \left\{
        M \in \R^{L \times L} \; \middle| \; M^T = M, \ \forall z \in \R^L \setminus \{0\} : z^T M z > 0
    \right\},
\end{equation*}
and based on that we define
\begin{equation*}
    \Pe := \left\{
        M \in \Pp \; \middle| \; \forall i \in \indexset[L] : \sigma(M)_i \geq \varepsilon^2
    \right\},
\end{equation*}
which is the set of all symmetric and positive definite matrices with eigenvalues of at least $\varepsilon^2$.
The indicator function for hyperspectral image segmentation from \cite{CoBaBe22} is
\begin{equation*}
    f_{l}(x; \mu_l, \Sigma_l) = \sqrt{\left( g(x) - \mu_l \right)^T \Sigma_l^{-1} \left( g(x) - \mu_l \right) + \eta} + \log \det \Sigma_l,
\end{equation*}
where $\mu_l \in \R^L$ and $\Sigma_l \in \Pe$.
The variables $\mu_l$ and $\Sigma_l$ are taking the roles of estimates of the mean and the covariance of the spectral distribution of the $l$-th segment, respectively.
We will use the notation
\begin{equation*}
    \Vert z \Vert_{\Sigma_l^{-\!1}\!\!, \eta} := \sqrt{z^T \Sigma_l^{-1} z + \eta}
\end{equation*}
for $z\in \R^L$, because the square root term can be seen as a regularized norm on $\R^L$, since $z \mapsto \sqrt{z^T M z}$ is a norm for any positive definite matrix $M$.
Hence, we have
\begin{equation*}
    f_{l}(x; \mu_l, \Sigma_l) = \Vert g(x) - \mu_l \Vert_{\Sigma_l^{-\!1}\!\!, \eta} + \log \det \Sigma_l.
\end{equation*}
The restriction of the eigenvalues of $\Sigma_l$ to at least $\varepsilon^2$ was introduced to ensure the invertibility of $\Sigma_l$, which also implies the feasibility of $\log \det \Sigma_l$.
As we will see in \cref{subsec:epsAMS-unboundedness-J-0}, this restriction cannot be dropped if we want to keep the feasibility and the guarantee that minimizers exist.
Defining $X_\varepsilon := \AAA \times \left(\R^L\right)^k \times \left(\Pe\right)^k$ and
\begin{equation*}
    \bOO := \left(\OO_1, \dots, \OO_k \right)\in \AAA,
    \bmu := \left(\mu_1, \dots, \mu_k \right)\in \big(\R^L\big)^k,
    \bSigma := \left(\Sigma_1, \dots, \Sigma_k \right)\in \left(\Pe\right)^k,
\end{equation*}
the objective function for hyperspectral image segmentation is
\begin{equation}\label{eq:definition-J-epsilon}
    \begin{split}
        J_{\varepsilon} \colon X_\varepsilon &\to \R,\\
        (\bOO, \bmu, \bSigma) &\mapsto \sum_{l=1}^k \left[
            \int_{\OO_l} \Vert g(x) - \mu_l \Vert_{\Sigma_l^{-\!1}\!\!, \eta} + \log \det \Sigma_l \dx + \lambda \perIm{\OO_l}\right].
    \end{split}
\end{equation}
By optimizing $J_\varepsilon$ over all parameters, we obtain a segmentation and estimates for the means $\mu_l$ and covariances $\Sigma_l$ for $l \in \indexset[k]$ from the data.

\subsection{Existence of minimizers of \texorpdfstring{$\varepsilon$}{epsilon}AMS}\label{subsec:epsAMS-existence-minimizers}
In this section, we use the \emph{direct method in the calculus of variations} to show %
that minimizers of $J_\varepsilon$ exist.
The direct method consists of four steps:
\begin{enumerate}
    \item finding a lower bound
    \begin{equation*}
        \underline{J_\varepsilon} := \inf_{(\bOO, \bmu, \bSigma) \in X_\varepsilon} J_\varepsilon[\bOO, \bmu, \bSigma] > - \infty,
    \end{equation*}
    \item choosing a minimizing sequence $(\bOO_n, \bmu_n, \bSigma_n)_{n\in \N} \subseteq X_\varepsilon$, i.e., the sequence fulfills $J_{\varepsilon}[\bOO_n, \bmu_n, \bSigma_n] \overset{n \to \infty}{\longrightarrow} \underline{J_\varepsilon}$,
    \item constructing a subsequence $(\bOO_{n_m}, \bmu_{n_m}, \bSigma_{n_m})_{m\in \N}$ and an element\newline $(\bOO^*, \bmu^*, \bSigma^*) \in X_\varepsilon$ such that
    \begin{equation}\label{eq:lower-semicontinuity-minimizing-sequence}
        J_\varepsilon[\bOO^*, \bmu^*, \bSigma^*] \leq \liminf_{m\to \infty} J_\varepsilon[\bOO_{n_m}, \bmu_{n_m}, \bSigma_{n_m}],
    \end{equation}
    \item concluding that $(\bOO^*, \bmu^*, \bSigma^*)$ is a minimizer of $J_\varepsilon$ due to
    \begin{equation}\label{eq:chain-of-inequalities}
        \begin{split}
            \underline{J_\varepsilon} &= \lim_{m \to \infty} J_\varepsilon[\bOO_{n_m}, \bmu_{n_m}, \bSigma_{n_m}] = \liminf_{m\to \infty} J_\varepsilon[\bOO_{n_m}, \bmu_{n_m}, \bSigma_{n_m}]\\
            &\geq J_\varepsilon[\bOO^*, \bmu^*, \bSigma^*] \geq \underline{J_\varepsilon}.
        \end{split}
    \end{equation}
\end{enumerate}
Before stating and proving the first main result, we state some definitions and lemmas.
Note that we use \emph{convergence in measure} (as defined below) when working with sequences of measurable, bounded sets.
\begin{definition}\label{def:convergence-in-measure}
    Let $(A_n)_{n\in\N}$ be a sequence of measurable, bounded subsets of $\R^d$ and $A \subseteq \R^d$ a measurable, bounded set.
    We say that \emph{$(A_n)_{n\in\N}$ converges to $A$ in measure in $\imdom$} if the Lebesgue measure of the symmetric difference in $\imdom$ converges to 0, i.e.,
    \begin{equation*}
        \lebesgue{(A_n \bigtriangleup A) \cap \imdom} \to 0
    \end{equation*}
    for $n\to \infty$.
    We use the notation: $A_n \to A$ for $n\to\infty$.
\end{definition}
\noindent
This implies that also the sequence of Lebesgue measures of a convergent sequence of sets is convergent:
\begin{lemma}\label{lem:convergence-of-measures}
    Let $(A_n)_{n\in\N}$ be a sequence of measurable, bounded subsets of $\R^d$ and $A \subseteq \R^d$ measurable and bounded.
    If $(A_n)_{n\in\N}$ converges to $A$ in measure in $\imdom$, then also the measures of $(A_n)_{n\in\N}$ restricted to $\imdom$ converge to the measure of $A$ restricted to $\imdom$ in $\R$, i.e.,
    \begin{equation*}
        \lebesgue{A_n \cap \imdom} \to \lebesgue{A \cap \imdom} \text{ for } n\to \infty.
    \end{equation*}
\end{lemma}
\begin{proof}
    We have
    \begin{equation*}
    \begin{split}
        &\vert \lebesgue{A_n \cap \imdom} - \lebesgue{A\cap \imdom} \vert
        = \left\vert \int_{\imdom} \chi_{A_n}(x) \dx - \int_{\imdom} \chi_{A}(x) \dx \right\vert\\
        &\leq  \int_{\imdom} \vert \chi_{A_n}(x) - \chi_{A}(x) \vert \dx
        = \Vert \chi_{A_n} - \chi_{A} \Vert_{L^1(\imdom)}
        = \lebesgue{(A_n\bigtriangleup A) \cap \imdom} \overset{n \to \infty}{\longrightarrow} 0.
    \end{split}
    \end{equation*}
    The last equality is true because the symmetric difference is equal to the $L^1$-difference of the corresponding characteristic functions.
\end{proof}
\begin{lemma}\label{lem:pe-closed}
    Let $\varepsilon > 0$.
    The set $\Pe$ is a closed subset of $\R^{L\times L}$.
\end{lemma}
\begin{proof}
    The transpose depends continuously on the matrix as a linear function on a finite-dimensional space.
    Moreover, the eigenvalues of a matrix vary continuously with the matrix entries \cite[3.1.2]{Or90}.
    Then, the closedness of $\Pe$ follows since continuous equality and non-strict inequality conditions are preserved in the limit.
\end{proof}
\begin{lemma}\label{lem:estimate-2-norm}
    Let $M \in \R^{L\times L}$ be a symmetric, positive semi-definite matrix with largest eigenvalue $\sigma(M)_1$ and smallest eigenvalue $\sigma(M)_L$.
    Then, %
    \begin{equation*}
        \sigma(M)_L \Vert z \Vert_2^2 \leq z^T M z \leq \sigma(M)_1 \Vert z \Vert_2^2\ \text{ for all }z \in \R^L.
    \end{equation*}
\end{lemma}
\noindent
The lemma follows from the principal axis theorem.
\begin{theorem}\label{th:ms-existence-minimizer}
    The functional $J_\varepsilon$, cf. \eqref{eq:definition-J-epsilon}, has a minimizer in $X_\varepsilon$.
\end{theorem}
\begin{proof}
    We use the \emph{direct method in the calculus of variations}.%

    To show that $J_\varepsilon$ has a lower bound, we rewrite the functional
    \begin{equation}\label{eq:functional-only-first-integral}
        J_\varepsilon[\bOO, \bmu, \bSigma]
        = \sum_{l=1}^k \left[
            \int_{\OO_l} \Vert g(x) - \mu_l \Vert_{\Sigma_l^{-\!1}\!\!, \eta} \dx + \lebesgue{\OO_l} \log \det \Sigma_l
            + \lambda \perIm{\OO_l}\right].
    \end{equation}
    For each $l\in \{1,\dots, k\}$, we have
        $\Vert g(x) - \mu_l \Vert_{\Sigma_l^{-\!1}\!\!, \eta} > 0$
    because $\eta > 0$ and $\Sigma_l$ is positive definite, implying that its inverse is positive definite, too.
    Moreover, it holds
    \begin{equation}\label{eq:lower-bound-log-det-sigma}
        \lebesgue{\OO_l} \log \det \Sigma_l = \lebesgue{\OO_l}\log \prod_{i=1}^{L} \sigma(\Sigma_l)_i \geq \lebesgue{\OO_l} \log \prod_{i=1}^{L} \varepsilon^2 = \lebesgue{\OO_l} \log\varepsilon^{2L} > -\infty,
    \end{equation}
    because $\Sigma_l \in \Pe$, the function $\log$ is monotonically increasing and for $\lebesgue{\OO_l}$ we have $\infty > \lebesgue{\imdom} \geq \lebesgue{\OO_l} \geq 0$ since $\lebesgue{\cdot}$ is a measure and $\OO_l \subseteq \imdom$.
    The last term in \eqref{eq:functional-only-first-integral}, namely $\lambda \perIm{\OO_l} = \lambda \tv{\chi_{\OO_l}}$, is non-negative because $\lambda > 0$ and $\tv{\chi_{\OO_l}}\geq 0$.
    Hence, we obtain
    \begin{equation*}
        J_\varepsilon[\bOO, \bmu, \bSigma] \geq \underline{c}\ \text{ for all } \left(\bOO, \bmu, \bSigma\right) \in X_\varepsilon,
    \end{equation*}
    for some $\underline{c} \in \R$ with $\underline{c} > - \infty$.
    This implies the existence of a minimizing sequence 
    \begin{equation*}
        \left(\bOO_n, \bmu_n, \bSigma_n \right)_{n\in\N} \subseteq X_\varepsilon.
    \end{equation*}
    In particular, the evaluations of the minimizing sequence are bounded as a convergent sequence of real numbers, i.e., there is a constant $\overline{C} < \infty$ such that
    \begin{equation}\label{eq:min-sequence-bounded-evals}
        J_\varepsilon[\bOO_n, \bmu_n, \bSigma_n] \leq \overline{C}\text{ for all }n\in\N.
    \end{equation}
    Please note that we will need that $\overline{C}$ also bounds the summands inside the parentheses in \eqref{eq:functional-only-first-integral} for every $l\in \indexset[k]$ separately.
    However, if $\varepsilon < 1$, the term $\log\det \Sigma_l$ can be negative.
    In that case, $\overline{C}$ has to be chosen sufficiently large, e.g., by adding $\lebesgue{\imdom}|\log\varepsilon^{2L}|$, to fulfill that the summands are bounded above separately by $\overline{C}$.

    We now construct a subsequence of the minimizing sequence and an element $(\bOO^*, \bmu^*, \bSigma^*) \in X_\varepsilon$ that fulfill \eqref{eq:lower-semicontinuity-minimizing-sequence}.
    Please note that every subsequence that we consider on our way to the final subsequence will be again denoted by
        $\left(\bOO_n, \bmu_n, \bSigma_n \right)_{n\in\N}$
    for ease of notation.
    The procedure to arrive at the sought subsequence of the minimizing sequence is to go through $\left(\bOO_n\right)_{n \in \N}$, $\left(\bSigma_n \right)_{n\in\N}$ and $\left(\bmu_n\right)_{n \in \N}$ component-wise and show how to choose a subsequence based on the respective component satisfying the requirements. That is, we choose a subsequence of $\left(\bOO_n, \bmu_n, \bSigma_n \right)_{n\in\N}$ such that $(\OO_{n,1})_{n\in\N}$ meets our requirements.
    From that new sequence $\left(\bOO_n, \bmu_n, \bSigma_n \right)_{n\in\N}$, we choose a subsequence such that $(\OO_{n,2})_{n\in\N}$ behaves as desired.
    We continue choosing subsequences until $(\OO_{n,k})_{n\in\N}$ fulfills the requirements, meaning that we are done with the $\bOO_n$ part.
    In the same manner, we go through $\left(\bSigma_n \right)_{n\in\N}$ and $\left(\bmu_n\right)_{n \in \N}$ and arrive at a subsequence $\left(\bOO_n, \bmu_n, \bSigma_n \right)_{n\in\N}$ for which the requirements for all components are met.

    The first step is to show that the sequence $(\OO_{n,l})_{n\in\N}$ as part of our minimizing sequence is bounded for each $l \in \indexset[k]$.
    To this end, we consider the sequence of characteristic functions 
    $
        (\chi_{\OO_{n,l}})_{n\in\N} \subseteq \bv,
    $
    which is in $\bv$ since $\bOO_n = (\OO_{n,1}, \dots, \OO_{n,k}) \in \AAA$.
    First, $\partial \Omega$ is compact because it is bounded due to the boundedness of $\Omega$ and also closed since %
        $\partial \Omega = \overline{\Omega} \cap \overline{\R^d \setminus \Omega}$.
    Since $\Omega$ also has Lipschitz boundary, $\Omega$ is an extension domain by \cite[Prop. 3.21]{AmFuPa00}.
    Furthermore, $(\chi_{\OO_{n,l}})_{n\in\N}$ is bounded for every $l\in \indexset[k]$ because for every $n\in\N$, we find
    \begin{equation*}
    \begin{split}
        \Vert \chi_{\OO_{n,l}} \Vert_{\bv} &= \Vert \chi_{\OO_{n,l}} \Vert_{L^1(\imdom)} + \tv{\chi_{\OO_{n,l}}}\\
        &= \int_{\imdom} \vert \chi_{\OO_{n,l}}(x) \vert \dx + \perIm{\OO_{n,l}}
        \leq \lebesgue{\imdom} + \frac{\overline{C}}{\lambda}
        < \infty,
    \end{split}
    \end{equation*}
    as $\vert \chi_{\OO_{n,l}}(x) \vert \leq 1$ and $\lambda \perIm{\OO_{n,l}} \leq \overline{C}$ since $\OO_{n,l}$ is part of the minimizing sequence.
    Due to the boundedness of $\Vert \chi_{\OO_{n,l}} \Vert_{\bv}$ and $\Omega$ being a bounded extension domain, \cite[Th. 3.23]{AmFuPa00} guarantees the existence of a subsequence $(\chi_{\OO_{n,l}})_{n\in\N}$ that converges weakly* in the sense of \cite[Def. 3.11]{AmFuPa00} with limit $u_l^* \in \bv$.
    In particular, this implies that this sequence converges strongly in $L^1(\imdom)$ to the same limit.
    Moreover, we choose a subsequence that converges even point-wise a.e. to $u_l^*$, cf. %
    \cite[Lemma 3.22 (1)]{Al16}.
    This ensures that $u_l^*(x) \in \{0,1\}$ for almost every $x\in\imdom$ because $\chi_{\OO_{n,l}}(x) \to u_l^*(x)$ in $\R$ and $\chi_{\OO_{n,l}}(x) \in \{0,1\}$ for every $n\in\N$.
    Moreover, $\OO_l^*:=(u_l^*)^{-1}(\{1\})$ is measurable because $u_l^*$ is measurable as a function in $\bv$ and therefore the preimage of the measurable set $\{1\} \in \mathcal{B}(\R)$ under $u_l^*$ is also measurable. Furthermore, we have $\chi_{\OO_l^*}=u_l^*$.

    To see that $\bOO^* = (\OO_1^*, \dots, \OO_k^*)$ is a partition of $\Omega$, note that
    \begin{equation*}
        S := \left\{(a_1, \dots, a_k) \in \R^k \; \middle| \; {\textstyle\sum\nolimits_{l=1}^{k}} a_l = 1\right\}
    \end{equation*}
    is closed.
    Since we have chosen a subsequence $\left(\bOO_n, \bmu_n, \bSigma_n \right)_{n\in\N}$ such that the sequence $(\chi_{\OO_{n,l}})_{n\in\N}$ converges point-wise a.e. for every $l \in \indexset[k]$, we have for almost every $x \in \Omega$ a sequence $(\chi_{\OO_{n,1}}(x), \dots, \chi_{\OO_{n,k}} (x))_{n\in\N}$ in $S$ with limit $(\chi_{\OO_{1}^*}(x), \dots, \chi_{\OO_{k}^*} (x))$.
    It follows that $(\chi_{\OO_{1}^*}(x), \dots, \chi_{\OO_{k}^*} (x)) \in S$ for almost every $x \in \Omega$ by the closedness of $S$. In this sense, the partition property of $\bOO^*$ holds almost everywhere.
    In other words, the limit $\bOO^*$ fulfills the required properties to belong to $\AAA$ only \emph{almost everywhere} and resides therefore in the set
    \begin{equation*}
    \begin{split}
        \AAA' := &\left\{
            \left( \OO_{1}, \OO_{2}, \dots, \OO_{k}\right) \; \middle| \; \forall l\in \indexset[k]: \OO_{l} \subseteq \imdom \text{ measurable}, \vphantom{{\textstyle\bigcup\nolimits_{l=1}^k}}\right.\\
            &\left. \hphantom{\left\{\right.} \perIm{\OO_{l}} < \infty, \ \forall m \neq l: \lebesgue{\OO_{l} \cap \OO_{m}} = 0, \lebesgue{\imdom \setminus {\textstyle\bigcup\nolimits_{l=1}^k} \OO_{l}} = 0
        \right\},
    \end{split}
    \end{equation*}
    which is a superset of $\AAA$.
    Those $x \in \imdom$ without guaranteed point-wise convergence are contained in a null set.
    Thus, we can change $\bOO^*$ such that the properties necessary to belong to $\AAA$ are fulfilled as this does not change the value of $J_\varepsilon$.
    Hence, we can assume that $\bOO^* \in \AAA$.
    Moreover, we have the equivalence of local convergence in measure and convergence in measure in $\imdom$ because $\imdom$ is bounded \cite[Rem. 3.37]{AmFuPa00}.
    Thus, $\OO \mapsto \perIm{\OO}$ is lower semi-continuous w.r.t. convergence in measure in $\imdom$ by \cite[Prop. 3.38 (b)]{AmFuPa00}.
    Consequently, the chosen $(\bOO_{n})_{n\in\N}$ and $\bOO^*$ satisfy
    \begin{equation}\label{eq:subsequence-lower-semicontinuity-perimeter}
        \perIm{\OO_l^*} \leq \liminf_{n\to\infty} \perIm{\OO_{n,l}}\text{ for each }l\in \indexset[k].
    \end{equation}
    In order to obtain subsequences $(\bmu_n)_{n\in\N} \subseteq (\R^L)^k$ and $(\bSigma_n)_{n\in\N} \subseteq (\Pe)^k$ that satisfy our requirements, we have to distinguish for every $l\in \indexset[k]$ the cases (1) that in the limit, segment $l$ has measure greater $0$, i.e., $\lebesgue{\OO_l^*} > 0$; and (2) that in the limit, segment $l$ has measure $0$, i.e., $\lebesgue{\OO_l^*} = 0$.

    Case (1), i.e. $\lebesgue{\OO_l^*} > 0$: Because we have a minimizing sequence, it holds
    \begin{equation*}
        \lebesgue{\OO_{n,l}} \log \det \Sigma_{n,l} \leq \overline{C}\text{ for all }n\in\N.
    \end{equation*}
    Moreover, because of $\Sigma_{n,l} \in \Pe$, we find for every $n\in \N$ the estimate
    \begin{equation}\label{eq:estimate-log-det-Sigma}
        \begin{split}
            \log \det \Sigma_{n,l}
            &= \log {\textstyle\prod_{i=1}^{L}} \sigma(\Sigma_{n,l})_i\\
            &\geq \log \left(\sigma(\Sigma_{n,l})_1\left[\sigma(\Sigma_{n,l})_L\right]^{L-1}\right)
            \geq \log \sigma(\Sigma_{n,l})_1 + \log \varepsilon^{2(L - 1)}.
        \end{split}
    \end{equation}
    Since $(\chi_{\OO_{n,l}})_{n\in\N}\rightarrow\chi_{\OO_l^*}$ in $L^1(\imdom)$ and
    because the convergence of the characteristic functions in $L^1(\imdom)$ corresponds to the convergence of sets in measure which implies the convergence of the measures (cf. \cref{lem:convergence-of-measures}),
    we also have $\lebesgue{\OO_{n,l}}\to \lebesgue{\OO_l^*}=:c_l$.
    Thus, for fixed $\delta \in (0, c_l)$, there is $N_l \in \N$ such that for $n \geq N_l$, we have
    \begin{equation}\label{eq:bounds-measure-segment-l}
        0 < c_l - \delta \leq \lebesgue{\OO_{n,l}} \leq c_l + \delta.
    \end{equation}
    By going to another subsequence (discarding the first $N_l-1$ elements), we achieve that \eqref{eq:bounds-measure-segment-l} holds for all $n\in\mathbb{N}$.
    If $\log \det \Sigma_{n,l} \geq 0$, we find with \eqref{eq:estimate-log-det-Sigma}
    \begin{equation*}
    \begin{split}
        \overline{C} &\geq \lebesgue{\OO_{n,l}} \log \det \Sigma_{n,l}
        \geq (c_l - \delta) \log \sigma(\Sigma_{n,l})_1 + (c_l - \delta) \log \varepsilon^{2(L - 1)}\\
        \Rightarrow\ &\tfrac{\overline{C}}{c_l - \delta} - \log \varepsilon^{2(L - 1)} \geq \log \sigma(\Sigma_{n,l})_1\\
        \Rightarrow\ &\infty > \exp\left(\tfrac{\overline{C}}{c_l - \delta} - \log \varepsilon^{2(L - 1)}\right) \geq \sigma(\Sigma_{n,l})_1.
    \end{split}
    \end{equation*}
    If $\log \det \Sigma_{n,l} < 0$, we start with
    \begin{equation*}
        \overline{C} \geq \lebesgue{\OO_{n,l}} \log \det \Sigma_{n,l}
        \geq (c_l + \delta) \log \det \Sigma_{n,l},
    \end{equation*}
    and derive analogously
    \begin{equation*}
        \tfrac{\overline{C}}{c_l + \delta} - \log \varepsilon^{2(L - 1)} \geq \log \sigma(\Sigma_{n,l})_1\ 
        \Rightarrow\ \infty > \exp\left(\tfrac{\overline{C}}{c_l + \delta} - \log \varepsilon^{2(L - 1)}\right) \geq \sigma(\Sigma_{n,l})_1.
    \end{equation*}
    Hence, $(\Sigma_{n,l})_{n\in\N}$ is bounded in the spectral norm $\Vert \cdot \Vert_{2\to 2}$.
    In particular, this implies that there exists an $r > 0$ such that
    \begin{equation*}
        (\Sigma_{n,l})_{n\in\N} \subseteq \overline{B_r^{\Vert \cdot \Vert_{2\to 2}}(0)} \cap \Pe \subseteq \R^{L\times L},
    \end{equation*}
    with $\overline{B_r^{\Vert \cdot \Vert_{2\to 2}}(0)}$ being the closure of the ball with radius $r$ around the origin with respect to the spectral norm.
    The intersection of the closed ball and $\Pe$ (which is closed, cf. \cref{lem:pe-closed}) is bounded and closed.
    By the Heine-Borel theorem, it is compact as a subset of a finite-dimensional space.
    Consequently, we can choose a convergent subsequence $(\Sigma_{n,l})_{n\in\N} \subseteq \Pe$ that converges strongly to $\Sigma_l^* \in \Pe$.

    Now, we consider $(\mu_{n,l})_{n\in\N} \subseteq \R^L$.
    By the reverse triangle inequality, we get
    \begin{equation*}
        \Vert \mu_{n,l} - g(x)\Vert_2 \geq |\Vert \mu_{n,l} \Vert_2 - \Vert g(x) \Vert_2|\geq \Vert \mu_{n,l} \Vert_2 - \Vert g(x) \Vert_2,
    \end{equation*}
    for all $x \in \imdom$. Moreover, the eigenvalues of $\Sigma_{n,l}^{-1}$ are the reciprocals of the eigenvalues of $\Sigma_{n,l}$.
    Since $(\mu_{n,l})_{n\in\N}$ is part of the minimizing sequence, we have
    \begin{align*}
        \overline{C} \geq{}& \int_{\OO_{n,l}} \Vert g(x) - \mu_{n,l} \Vert_{\Sigma_{n,l}^{-\!1}, \eta} \dx\\
        \geq{}& \int_{\OO_{n,l}} \sqrt{\left(g(x) - \mu_{n,l}\right)^T \Sigma_{n,l}^{-1} \left(g(x) - \mu_{n,l}\right)} \dx\\
        \prescript{\text{\cref{lem:estimate-2-norm}}}{}{\geq}~& \int_{\OO_{n,l}} \sqrt{\tfrac{1}{\sigma(\Sigma_{n,l})_1}} \; \Vert g(x) - \mu_{n,l} \Vert_2 \dx
        \\
        \geq{}& \sqrt{\tfrac{1}{\sigma(\Sigma_{n,l})_1}} \biggl(\int_{\OO_{n,l}} \Vert \mu_{n,l} \Vert_2 \dx -  \int_{\OO_{n,l}} \Vert g(x) \Vert_2 \dx\biggr)\\
        \geq{}& \sqrt{\tfrac{1}{\sigma(\Sigma_{n,l})_1}} \biggl(\lebesgue{\OO_{n,l}} \; \Vert \mu_{n,l} \Vert_2 - \lebesgue{\OO_{n,l}} \; \Vert \Vert g \Vert_2 \Vert_{L^\infty(\OO_{n,l})}\biggr)\\
        \geq{}& \sqrt{\tfrac{1}{\sigma(\Sigma_{n,l})_1}} \biggl(\lebesgue{\OO_{n,l}} \; \Vert \mu_{n,l} \Vert_2 - \lebesgue{\imdom} \; \Vert \Vert g \Vert_2 \Vert_{L^\infty(\imdom)}\biggr).
    \end{align*}
    As a consequence, we find
    \begin{equation}\label{eq:upper-bound-lambda-mu}
        \sqrt{\sigma(\Sigma_{n,l})_1} \; \overline{C} + \lebesgue{\imdom} \; \Vert \Vert g \Vert_2 \Vert_{L^\infty(\imdom)} \geq \lebesgue{\OO_{n,l}} \; \Vert \mu_{n,l} \Vert_2.
    \end{equation}
    The left-hand side is bounded from above because $\imdom$ is bounded, $g \in L^\infty(\imdom, \R^L)$ and $\Vert \Sigma_{n,l} \Vert_{2\to 2}\leq\tilde C < \infty$ for all $n\in\N$ and a constant $\tilde{C} \in \R$, as shown above.
    Combining \eqref{eq:bounds-measure-segment-l} and \eqref{eq:upper-bound-lambda-mu}, we get
    \begin{equation*}
        \sqrt{\tilde{C}} \; \overline{C} + \lebesgue{\imdom} \; \Vert \Vert g \Vert_2 \Vert_{L^\infty(\imdom)}
        \geq \lebesgue{\OO_{n,l}} \; \Vert \mu_{n,l} \Vert_2
        \geq (c_l - \delta) \, \Vert \mu_{n,l} \Vert_2,
    \end{equation*}
    and the boundedness of $(\mu_{n,l})_{n\in\N}$ follows from division by $c_l - \delta$.
    With the same reasoning as before, using the Heine-Borel theorem, we can choose a strongly convergent subsequence with limit $\mu_l^* \in \R^L$.

    In order to show the inequality \eqref{eq:lower-semicontinuity-minimizing-sequence} for the $l$-th summand, we first note that the transpose %
    and inverse of a matrix vary continuously with their entries and that square root, determinant and logarithm are continuous functions.
    Furthermore, we use \eqref{eq:subsequence-lower-semicontinuity-perimeter} and that the sum and the product of functions of point-wise a.e. convergent sequences converge point-wise a.e., too.
    Since $(\OO_{n,l})_{n\in\N}$, $(\mu_{n,l})_{n\in\N} \subseteq \R^L$ and $(\Sigma_{n,l})_{n\in\N} \subseteq \Pe$ converge to $\OO_l^*$ (in measure and point-wise a.e. as characteristic functions), $\mu_{l}^* \in \R^L$ and $\Sigma_l^* \in \Pe$, respectively, we get for the $l$-th summand of the outer sum of $J_\varepsilon$ \eqref{eq:definition-J-epsilon}
    \begin{align}\label{eq:inequality-liminf-convergent-sequences}
        &\int_{\OO_l^*} \Vert g(x) - \mu_l^* \Vert_{\left(\Sigma_l^*\right)^{-\!1}\!\!, \eta} + \log \det \Sigma_l^* \dx + \lambda \perIm{\OO_l^*}\notag\\
        =& \int_{\imdom} \chi_{\OO_l^*}(x) \left(\Vert g(x) - \mu_l^* \Vert_{\left(\Sigma_l^*\right)^{-\!1}\!\!, \eta} + \log \det \Sigma_l^* \right) \dx + \lambda \perIm{\OO_l^*}\notag\\
        =& \int_{\imdom} \lim_{n\to\infty} \chi_{\OO_{n,l}}(x) \left(\Vert g(x) - \mu_{n,l} \Vert_{\Sigma_{n,l}^{-\!1}, \eta} + \log \det \Sigma_{n,l} \right) \dx
        + \lambda \perIm{\OO_l^*}\notag\\
        \leq& \liminf_{n\to\infty} \int_{\imdom} \chi_{\OO_{n,l}}(x) \left( \Vert g(x) - \mu_{n,l} \Vert_{\Sigma_{n,l}^{-\!1}, \eta}+ \log \det \Sigma_{n,l} \right) \dx \notag\\
        &+ \liminf_{n\to\infty}\lambda \perIm{\OO_{n,l}}\notag\\
        \leq& \liminf_{n\to\infty} \int_{\OO_{n,l}} \Vert g(x) - \mu_{n,l} \Vert_{\Sigma_{n,l}^{-\!1}, \eta} + \log \det \Sigma_{n,l} \dx + \lambda \perIm{\OO_{n,l}}.
    \end{align}
    The second equality uses that the integrand converges point-wise a.e. due the convergence properties shown before.
    The penultimate inequality follows from \eqref{eq:subsequence-lower-semicontinuity-perimeter} and Fatou's lemma \cite[Th. 1.20]{AmFuPa00}, which holds, in particular, for measurable and possibly negative functions that are bounded below by a function in $L^1(\imdom)$. Such an integrable lower bound exists due to the boundedness of $\Omega$ and $\log \det \Sigma_{n,l} \geq \log\varepsilon^{2L}$, cf. \eqref{eq:lower-bound-log-det-sigma}.

    Summing up the above, for $l\in \indexset[k]$ such that $\lebesgue{\OO_l^*} > 0$, we have constructed a convergent subsequence of the minimizing sequence and obtained elements $\OO_l^*$, $\mu_l^*$ and $\Sigma_l^*$ as the limit of the sequence that fulfill the inequality stated in \eqref{eq:lower-semicontinuity-minimizing-sequence}.

    We consider now case (2):
    First, for any $n\in\mathbb{N}$ with $\lebesgue{\OO_{n,l}} = 0$, we set $\mu_{n,l} := 0 \in \R^L$ and $\Sigma_{n,l} := I_L$,
    where $I_L$ denotes the identity matrix in $\R^{L \times L}$. Note that this does not change the value of $J_\varepsilon$ and the sequence keeps all properties shown before.
    We start by considering the second summand in \eqref{eq:functional-only-first-integral}.
    Since $\lebesgue{\OO_{n,l}} \to 0$ and $\log \det \Sigma_{n,l} \geq 2L \log \varepsilon$ for all $n\in\N$, cf. \eqref{eq:lower-bound-log-det-sigma}, two cases have to be considered:
    \begin{enumerate}
    \item If $\log \det \Sigma_{n,l}$ is bounded from above, i.e., $\log \det \Sigma_{n,l} \leq \alpha$ for some $\alpha < \infty$, it follows with \eqref{eq:estimate-log-det-Sigma}
    \begin{align*}
        \alpha &\geq \log \det \Sigma_{n,l}
        \geq \log \sigma(\Sigma_{n,l})_1 + \log \varepsilon^{2(L-1)}\\
        \Rightarrow\ \infty &> \exp\left[\alpha - \log \varepsilon^{2(L-1)}\right] \geq \sigma(\Sigma_{n,l})_1.
    \end{align*}
    Hence, in this case $(\Sigma_{n,l})_{n\in\N}$ is bounded in the spectral norm and therefore, by the Heine-Borel theorem, we can choose a convergent subsequence with limit $\Sigma_{l}^* \in \Pe$.
    As a consequence, we obtain
    \begin{align*}
        &\lebesgue{\OO_{l}^*} \log \det \Sigma_{l}^*
        = \int_{\imdom} \chi_{\OO_l^*}(x) \log \det \Sigma_{l}^* \dx\\
        &= \int_{\imdom} \lim_{n\to \infty} \chi_{\OO_{n,l}}(x)  \log \det \Sigma_{n,l} \dx
        \leq \liminf_{n\to \infty} \int_{\imdom} \chi_{\OO_{n,l}}(x) \log \det \Sigma_{n,l} \dx,
    \end{align*}
    again using point-wise a.e. convergence of the integrand and Fatou's lemma \cite[Th. 1.20]{AmFuPa00} for a sequence of functions with an integrable lower bound.
    \item If $\log \det \Sigma_{n,l}$ is not bounded from above, we do not have  control over $\Sigma_{n,l}$.
    Therefore, different from the typical use of the direct method, we do not try to find a convergent subsequence but just a (possibly non-convergent) subsequence $(\Sigma_{n,l})_{n\in\N}$ and an element $\Sigma_l^* \in \Pe$ such that \eqref{eq:lower-semicontinuity-minimizing-sequence} is fulfilled.
    Hence, if $\log \det \Sigma_{n,l}$ is not bounded from above, we choose a subsequence of $(\Sigma_{n,l})_{n\in\N}$ such that $\log \det \Sigma_{n,l}$ is monotonically increasing and non-negative (and diverging to $\infty$).
    Setting $\Sigma_l^* := \varepsilon^2 I_L \in \Pe$, we obtain
    \begin{equation*}
        \underbrace{\lebesgue{\OO_l^*}}_{=0} \log\det\Sigma_l^*
        =0 %
        \leq \liminf_{n\to \infty} \int_{\imdom} \underbrace{\chi_{\OO_{n,l}}(x)}_{\geq 0} \underbrace{\log\det\Sigma_{n,l}}_{\geq 0} \dx.
    \end{equation*}
    \end{enumerate}
    We will now consider the integral in \eqref{eq:functional-only-first-integral}.
    Also for $\left(\mu_{n,l}\right)_{n\in\N}$, we cannot hope to find a convergent subsequence in this case and will therefore only aim to satisfy  \eqref{eq:lower-semicontinuity-minimizing-sequence} with $\left(\mu_{n,l}\right)_{n\in\N}$ and an element $\mu_l^*\in \R^L$.
    Since the integrand is always positive by the positive definiteness of $\Sigma_{n,l}^{-1}$ and the positivity of $\eta$, the integral is obviously non-negative for every $n \in \N$.
    Combined with choosing $\mu_l^* := 0 \in \R^L$ and again using $|\OO_l^*|=0$, we obtain
    \begin{equation*}
        \int_{\OO_l^*} \Vert g(x) - \mu_l^* \Vert_{\left(\Sigma_l^*\right)^{-\!1}\!\!, \eta} \dx = 0
        \leq \liminf_{n\to \infty} \int_{\OO_{n,l}} \Vert g(x) - \mu_{n,l} \Vert_{\Sigma_{n,l}^{-\!1}, \eta} \dx.
    \end{equation*}
    Combining all considerations for the case $\lebesgue{\OO_l^*} = 0$ and using again \eqref{eq:subsequence-lower-semicontinuity-perimeter}, we get
    \begin{align*}
        &\int_{\OO_l^*} \Vert g(x) - \mu_l^* \Vert_{\left(\Sigma_l^*\right)^{-\!1}\!\!, \eta} \dx + \lebesgue{\OO_l^*} \log \det \Sigma_l^* + \lambda \perIm{\OO_l^*}\\
        \leq& \liminf_{n\to \infty} \int_{\OO_{n,l}} \Vert g(x) - \mu_{n,l} \Vert_{\Sigma_{n,l}^{-\!1}, \eta} \dx
        + \int_{\imdom} \chi_{\OO_{n,l}}(x) \log \det \Sigma_{n,l} \dx\\
        &+ \lambda \perIm{\OO_{n,l}}.
    \end{align*}
    In total, we got a minimizing sequence
    $
        (\bOO_{n}, \bmu_n, \bSigma_n)_{n\in\N} \subseteq X_\varepsilon
    $
    and 
    $
        (\bOO^*, \bmu^*, \bSigma^*) \in X_\varepsilon
    $
    with $\bOO^* = (\OO_1^*,\dots,\OO_k^*)$, $\bmu^* = (\mu_1^*, \dots, \mu_k^*)$ and $\bSigma^* = (\Sigma_1^*, \dots, \Sigma_k^*)$
    such that
    \begin{align*}
        &\int_{\OO_l^*} \Vert g(x) - \mu_l^* \Vert_{\left(\Sigma_l^*\right)^{-\!1}\!\!, \eta} \dx + \lebesgue{\OO_l^*} \log \det \Sigma_l^* + \lambda \perIm{\OO_l^*}\\
        \leq& \liminf_{n\to \infty} \int_{\OO_{n,l}} \Vert g(x) - \mu_{n,l} \Vert_{\Sigma_{n,l}^{-\!1}, \eta} \dx + \lebesgue{\OO_{n,l}} \log \det \Sigma_{n,l}
        + \lambda \perIm{\OO_{n,l}}
    \end{align*}
    is true for every $l \in \indexset[k]$.
    Using again that the sum of the limit inferiors of two sequences is smaller or equal to the limit inferior of the sum of the sequences, the inequality \eqref{eq:lower-semicontinuity-minimizing-sequence} is fulfilled.
    This implies \eqref{eq:chain-of-inequalities}, i.e., $(\bOO^*, \bmu^*, \bSigma^*)$ is a minimizer of $J_\varepsilon$.
\end{proof}

\subsection{Gamma convergence of \texorpdfstring{$J_\varepsilon$}{J epsilon} when \texorpdfstring{$\varepsilon \to 0$}{epsilon to zero}}\label{subsec:epsAMS-gamma-convergence}
In \cref{subsec:epsAMS-existence-minimizers}, we have seen that $\varepsilon$ is crucial in several parts of the proof of \cref{th:ms-existence-minimizer}; namely, to show that a lower bound of $J_\varepsilon$ exists, to choose subsequences of the minimizing sequence that fulfill \eqref{eq:lower-semicontinuity-minimizing-sequence}, and, most importantly, to achieve that we optimize w.r.t. $\Sigma_l$ over a closed set $\Pe$.
However, introduced to ensure that the inverse of $\Sigma_l$ exists, the parameter $\varepsilon$ is supposed to be chosen close to $0$ and the question arises what happens if $\varepsilon \to 0$.
The notion of \gammaconvergence{} helps answering this question.
In particular, we will show the \gammaconvergence{} of $J_\varepsilon$ to a \gammalimit{}.

Let $X := \AAA \times \left(\R^L\right)^k \times \Pp^k$.
We use the definition of \gammaconvergence{} given in \cite[Def. 1.5, Def. 1.45]{Br02} and also refer to \cite{Br02} for general information about \gammaconvergence{}.
It requires a notion of convergence for sequences in $X$, which we choose as follows.
\begin{definition}\label{def:gamma-conv-notion-convergence}
    We say that a sequence \emph{$(\bOO_n, \bmu_n, \bSigma_n)_{n \in \N} \subseteq X$ converges weakly* to $(\bOO, \bmu, \bSigma) \in X$} if for every $l \in \indexset[k]$ it holds
    \begin{enumerate}
        \item $\chi_{\OO_{n,l}} \toweakstar \chi_{\OO_l}$ in $\bv$ in the sense of \cite[Def. 3.11]{AmFuPa00},
        \item $\mu_{n,l} \to \mu_l$ in $\R^L$,
        \item $\Sigma_{n,l} \to \Sigma_{l}$ in $\R^{L \times L}$,
    \end{enumerate}
    when $n\to \infty$.
    We write: $(\bOO_n, \bmu_n, \bSigma_n) \toweakstar (\bOO, \bmu, \bSigma)$.
\end{definition}
\noindent
Please note that if $\imdom$ is sufficiently regular, weak* convergence in $\bv$ in the sense of \cite[Def. 3.11]{AmFuPa00} corresponds to weak* convergence in the usual sense \cite[Rem. 3.12]{AmFuPa00}.
In that case, \cref{def:gamma-conv-notion-convergence} describes, in fact, the usual weak* convergence in $X$.
Obviously, we have $X_\varepsilon \subseteq X$.
It is necessary to extend $J_\varepsilon$ for $\varepsilon > 0$ from $X_\varepsilon$ to $X$ as the domain and from $\R$ to $\RExt := \R\cup \{+\infty\}$ as the image.
From now on, $J_\varepsilon$ accepts any symmetric and positive definite matrix as input and penalizes it with a functional value of $\infty$ if this matrix is not in $\Pe$.
Hence, it is given as
\begin{align}
    \label{eq:extension-J-epsilon}
        J_{\varepsilon} \colon X &\to \RExt,\\
        (\bOO, \bmu, \bSigma) &\mapsto \sum_{l=1}^k \left[
            \int_{\OO_l} \Vert g(x) - \mu_l \Vert_{\Sigma_l^{-\!1}\!\!, \eta} + \log \det \Sigma_l \dx
            + \lambda \perIm{\OO_l} + I_{\Pe}(\Sigma_l)\right]\notag.
\end{align}
The function $I_{\Pe}(\Sigma_l)$ for the set $\Pe$ is defined as
\begin{equation*}
    I_{\Pe}(\Sigma_l) := 
        \begin{cases}
            0 &\quad \text{if } \Sigma_l \in \Pe,\\
            \infty &\quad \text{if } \Sigma_l \notin \Pe.
        \end{cases}
\end{equation*}
Apparently, \cref{th:ms-existence-minimizer} also ensures the existence of minimizers for the extended $J_\varepsilon$.

To establish the \gammaconvergence{} of $(J_\varepsilon)_{\varepsilon > 0}$ to a \gammalimit{} $J_0 \colon X \to \RExt$, two inequalities have to be shown for each sequence $(\varepsilon_n)_{n\in\N}$ with $\varepsilon_n > 0$ for all $n\in\N$ and $\varepsilon_n \to 0$ when $n\to \infty$ and each $(\bOO, \bmu, \bSigma)\in X$:
\begin{enumerate}
    \item[(a)] \emph{lim inf inequality}:
    For every weakly* convergent sequence $(\bOO_n, \bmu_n, \bSigma_n)_{n \in \N} \subseteq X$ with limit $(\bOO, \bmu, \bSigma)$ it holds
    \begin{equation*}
        J_0[\bOO, \bmu, \bSigma] \leq \liminf_{n\to \infty} J_{\varepsilon_n} [\bOO_n, \bmu_n, \bSigma_n],
    \end{equation*}
    \item[(b)] \emph{lim sup inequality}:
    There is a weakly* convergent sequence $(\bOO_n, \bmu_n, \bSigma_n)_{n \in \N} \subseteq X$ with limit $(\bOO, \bmu, \bSigma)$ such that
    \begin{equation*}
        J_0[\bOO, \bmu, \bSigma] \geq \limsup_{n\to \infty} J_{\varepsilon_n} [\bOO_n, \bmu_n, \bSigma_n].
    \end{equation*}
\end{enumerate}
If these two conditions are satisfied, then $(J_\varepsilon)_{\varepsilon > 0}$ \gammaconverges{} to $J_0$, which will be shown next, including the specification of the limit $J_0$.
\begin{theorem}\label{th:gamma-convergence-J-eps}
    Consider the family of functionals $(J_\varepsilon)_{\varepsilon > 0}$, as defined in \eqref{eq:extension-J-epsilon} with $J_\varepsilon \colon X \to \RExt$.
    Define the functional $J_0$ as
    \begin{equation}\label{eq:def-J-0-formula}
        \begin{split}
            J_0 \colon X &\to \RExt,\\
            (\bOO, \bmu, \bSigma) &\mapsto \sum_{l=1}^k \left[\int_{\OO_l} \Vert g(x) - \mu_l \Vert_{\Sigma_l^{-\!1}\!\!, \eta} + \log \det \Sigma_l \dx + \lambda \perIm{\OO_l}\right].
        \end{split}
    \end{equation}
    Then, $(J_{\varepsilon})_{\varepsilon > 0}$ \gammaconverges{} to $J_0$ when $\varepsilon \to 0$ with respect to the weak* convergence in $X$ (cf. \cref{def:gamma-conv-notion-convergence}).
\end{theorem}
\begin{proof}
    Let $(\bOO, \bmu, \bSigma) \in X$ and $(\varepsilon_n)_{n\in \N}$ such that $\varepsilon_n > 0$ and $\varepsilon_n \stackrel{n \to \infty}{\longrightarrow} 0$.
    We start by showing the lim inf inequality.
    Let $(\bOO_n, \bmu_n, \bSigma_n)_{n \in \N} \subseteq X$ be a sequence such that $(\bOO_n, \bmu_n, \bSigma_n) \toweakstar (\bOO, \bmu, \bSigma)\in X$.
    Then, $(J_{\varepsilon_n}[\bOO_n, \bmu_n, \bSigma_n])_{n\in\N} \subseteq \RExt$ is the sequence of the corresponding functional values.
    Without loss of generality, we can assume
    \begin{equation*}
        \lim_{n\to \infty} J_{\varepsilon_{n}}[\bOO_{n}, \bmu_{n}, \bSigma_{n}] = \liminf_{n\to\infty} J_{\varepsilon_n}[\bOO_n, \bmu_n, \bSigma_n],
    \end{equation*}
    otherwise we select a subsequence of $(\bOO_n, \bmu_n, \bSigma_n)_{n \in \N}$ such that the corresponding functional values converge to $\liminf_{n\to \infty} J_{\varepsilon_n}[\bOO_n, \bmu_n, \bSigma_n]$.
    First, consider the term $I_{\Pe}(\Sigma_l)$ of \eqref{eq:extension-J-epsilon}.
    For any $\varepsilon > 0$ and $\Sigma \in \Pp$, we have $I_{\Pe}(\Sigma) \geq 0$ and $I_{\Pp}(\Sigma) = 0$, i.e.,
    \begin{equation}\label{eq:inequality-gamma-conv-liminf-indicator}
        I_{\Pp}(\Sigma_l) = 0 \leq \liminf_{n \to \infty} I_{\mathcal{P}_{\varepsilon_{n}}}(\Sigma_{n,l})\text{ for each }l\in \indexset[k].
    \end{equation}
    We now focus on the remaining terms of \eqref{eq:extension-J-epsilon}.
    Since $(\bOO_{n}, \bmu_{n}, \bSigma_{n}) \toweakstar (\bOO, \bmu, \bSigma)$, we have that $\chi_{\OO_{n,l}} \toweakstar \chi_{\OO_l}$ in $\bv$ and therefore $\chi_{\OO_{n,l}} \to \chi_{\OO_l}$ in $L^1(\imdom)$ for each $l \in \indexset[k]$, guaranteed by \cite[Prop. 3.13]{AmFuPa00}.
    In particular, we choose a subsequence of $(\bOO_{n}, \bmu_{n}, \bSigma_{n})_{n\in\N}$ such that $(\chi_{\OO_{n,l}})_{n\in\N}$ for $l=1$ converges point-wise a.e. to $\chi_{\OO_l}$, from which we select another subsequence such that $(\chi_{\OO_{n,l}})_{n\in\N}$ for $l=2$ converges point-wise a.e. to $\chi_{\OO_l}$, and so on until $l=k$.
    The existence of such subsequences follows from \cite[Lemma 3.22 (1)]{Al16}.
    The resulting subsequence is again denoted with $(\bOO_{n}, \bmu_{n}, \bSigma_{n})_{n\in\N}$.
    Analogously to the derivation in \eqref{eq:inequality-liminf-convergent-sequences}, we get
    \begin{align*}
        &\int_{\OO_l} \Vert g(x) - \mu_l \Vert_{\Sigma_l^{-\!1}\!\!, \eta} + \log \det \Sigma_l \dx + \lambda \perIm{\OO_l}\\
        \leq& \liminf_{n\to\infty} \int_{\OO_{n,l}} \Vert g(x) - \mu_{n,l} \Vert_{\Sigma_{n,l}^{-\!1}, \eta} + \log \det \Sigma_{n,l} \dx + \lambda \perIm{\OO_{n,l}}
    \end{align*}
    for each $l\in \indexset[k]$. Combined with \eqref{eq:inequality-gamma-conv-liminf-indicator}, we conclude
    \begin{equation*}
        J_0[\bOO, \bmu, \bSigma]
        \leq \liminf_{n\to\infty} J_{\varepsilon_n}[\bOO_n, \bmu_n, \bSigma_n].
    \end{equation*}
    This shows the lim inf inequality.

    To show the lim sup inequality, let $(\bOO, \bmu, \bSigma) \in X$ and $(\varepsilon_n)_{n\in\N} \subseteq \R$ such that $\varepsilon_n > 0$ and $\varepsilon_n \to 0$ when $n\to \infty$.
    We choose the constant sequence $(\bOO_n, \bmu_n, \bSigma_n)_{n\in\N}$ with $(\bOO_n, \bmu_n, \bSigma_n) = (\bOO, \bmu, \bSigma)$ for every $n\in \N$.
    This sequence obviously converges weakly* to $(\bOO, \bmu, \bSigma)$.
    Moreover, note that if $\varepsilon'' \geq \varepsilon' > 0$, then it follows $\mathcal{P}_{\varepsilon''} \subseteq \mathcal{P}_{\varepsilon'}$.
    Combined with $\varepsilon_n \to 0$ and $\Sigma_l\in\Pp$, there exists an $N\in\N$ such that $\Sigma_{n,l} = \Sigma_l \in \mathcal{P}_{\varepsilon_n}$ for every $l \in \indexset[k]$ and every $n \geq N$.
    As a consequence, we obtain that only for finitely many $n \in \N$, namely for at most $N-1$, it can hold
    \begin{equation*}
        J_{\varepsilon_n}[\bOO_n, \bmu_n, \bSigma_n] = J_{\varepsilon_n}[\bOO, \bmu, \bSigma] = \infty.
    \end{equation*}
    For all $n \geq N$, it must hold
        $I_{\mathcal{P}_{\varepsilon_{n}}}(\Sigma_{n,l}) = I_{\mathcal{P}_{\varepsilon_{n}}}(\Sigma_{l}) = 0,$
    and therefore
    \begin{equation*}
        J_{\varepsilon_n}[\bOO_n, \bmu_n, \bSigma_n] = J_{\varepsilon_n}[\bOO, \bmu, \bSigma] = J_0[\bOO, \bmu, \bSigma].
    \end{equation*}
    We can conclude that the sequence of evaluations $(J_{\varepsilon_n}[\bOO_n, \bmu_n, \bSigma_n])_{n\in\N} \subseteq \RExt$ is convergent with limit $J_0[\bOO, \bmu, \bSigma]$.
    This results in
    \begin{equation*}
        J_0[\bOO, \bmu, \bSigma] = \lim_{n \to \infty} J_{\varepsilon_n}[\bOO_n, \bmu_n, \bSigma_n] = \limsup_{n\to\infty} J_{\varepsilon_n}[\bOO_n, \bmu_n, \bSigma_n]
    \end{equation*}
    and shows the lim sup inequality.
\end{proof}
\noindent
Please note that $J_0$ is weakly* lower semi-continuous as a \gammalimit{} \cite{Br02}.

\subsection{Loss of closedness of \texorpdfstring{$\Pe$}{Pepsilon} when \texorpdfstring{$\varepsilon \to 0$}{epsilon to zero}}\label{subsec:epsAMS-unboundedness-J-0}
The \gammaconvergence{} of a family of functionals to a \gammalimit{} implies under certain conditions also the convergence of minimizers of the family of functionals to a minimizer of the \gammalimit{} (cf. \cite[Def. 1.19, Th. 1.21]{Br02}).
However, these conditions are not fulfilled in our case.
In particular, we will now see an explicit example where a minimizer does not exist if we do not use the regularization with $\varepsilon > 0$, i.e., if we allow $\Sigma_l \in \Pp$ instead of restricting it to $\Sigma_l \in \Pe$.

While $\bOO \in \AAA$ and $\bmu \in \left(\R^L\right)^k$ do not cause any trouble as $J_0$ is bounded below with respect to these parameters, the data term of $J_0$, which is the square root in \eqref{eq:def-J-0-formula}, and the term $\log \det \Sigma_l$ do not always balance each other. For certain images $g$, the latter can make the functional arbitrarily small because certain terms of the former vanish.
To see this, we consider for $l\in \indexset[k]$ the corresponding parts of $J_0$ from 
\eqref{eq:def-J-0-formula} and let $\Sigma_l = V_l D_l V_l^T$ be the eigendecomposition of $\Sigma_l$ with an orthogonal matrix $V_l \in \R^{L \times L}$ and a diagonal matrix $D_l \in \R^{L\times L}$ with the eigenvalues of $\Sigma_l$ on its diagonal.
For $\Sigma_l^{-1}$, it holds $\Sigma_l^{-1} = V_l D_l^{-1} V_l^T$.
We obtain
\begin{align}\label{eq:J-0-formula-segment-l}
    &\int_{\OO_l} \Vert g(x) - \mu_l \Vert_{\Sigma_l^{-\!1}\!\!, \eta} + \log \det \Sigma_l \dx\notag\\
    =& \int_{\OO_l} \sqrt{\left[V_l^T\left(g(x) - \mu_l\right)\right]^T D_l^{-1} \left[V_l^T\left(g(x) - \mu_l\right)\right] + \eta}\dx + \lebesgue{\OO_l} \log \left(\prod_{i=1}^{L} \sigma(\Sigma_l)_i\right)\notag\\
    =& \int_{\OO_l} \sqrt{\sum_{i=1}^{L} \frac{\left[V_l^T(g(x) - \mu_l)\right]_i^2}{\sigma(\Sigma_l)_i} + \eta}\dx + \lebesgue{\OO_l} \sum_{i=1}^{L} \log \sigma(\Sigma_l)_i.
\end{align}
The data term benefits from $\sigma(\Sigma_l)_i$ for $i\in \indexset[L]$ being very large as this makes the fraction small and approach $0$.
On the contrary, the $\log \det \Sigma_l$ term benefits from bringing $\sigma(\Sigma_l)_i$ as close to $0$ as possible.
In this sense, the two terms are antagonists.

An example in which they do not balance each other is as follows:
Let $g \in L^\infty(\imdom, \R^L)$ such that $g(\imdom)$ is $(L - m)$-dimensional for a fixed $m \in \indexset[L - 1]$.
Since $\imdom$ is bounded, we get $g \in L^2(\imdom, \R^L)$.
Fix $\OO_l \subseteq \imdom$ and $\mu_l \in \R^L$.
We take an orthonormal basis of $\operatorname{span}(g(\imdom))$ and complete it to an orthonormal basis of $\R^L$ that we write as columns into a matrix $V_l \in \R^{L\times L}$.
This basis consists of vectors $v_1, \dots, v_{L-m} \in \R^L$ that are a basis of $\operatorname{span}(g(\imdom))$ and vectors $v_{L - m + 1},\dots, v_L \in \R^L$ that are a basis of $\operatorname{span}(g(\imdom))^\perp$.
We choose a diagonal matrix $D_l \in \R^{L\times L}$ with entries $\sigma_1,\dots,\sigma_L > 0$ on its diagonal and set $\Sigma_l := V_l D_l V_l^T$.
It holds $\Sigma_l \in \Pp$.
Then, $\langle\cdot,\cdot\rangle\colon \R^L \times \R^L \to \R, (y, z) \mapsto y^T \Sigma_l^{-1} z$ defines an inner product on $\R^L$.
We have $\langle g(x), v_{L - m + i}\rangle = 0$ for almost every $x \in \imdom$ and every $i\in \indexset[m]$.
Let $P\colon \R^L \to \R^L$ denote the orthogonal projection with respect to $\langle \cdot, \cdot \rangle$ onto the subspace spanned by $v_{L-m+1}, \dots, v_{L}$ and $I$ the identity operator.
We find
\begin{align*}
    &\Vert g(x) - \mu_l \Vert_{\Sigma_l^{-\!1}\!\!, \eta}^2
    = \Vert g(x) - (I - P)\mu_l - P \mu_l \Vert_{\Sigma_l^{-\!1}\!\!, \eta}^2\\
    ={}& \langle g(x) - (I - P)\mu_l, g(x) - (I - P)\mu_l\rangle + \langle P\mu_l, P \mu_l \rangle + \eta\\
    ={}& \Vert g(x) - (I - P)\mu_l \Vert_{\Sigma_l^{-\!1}\!\!, \eta}^2 + \langle P\mu_l, P\mu_l \rangle
    \geq \Vert g(x) - (I - P)\mu_l \Vert_{\Sigma_l^{-\!1}\!\!, \eta}^2.
\end{align*}
The third equality holds by the linearity of the inner product and the orthogonality of $g(\imdom)$ and $v_{L - m + 1}, \dots, v_L$.
The inequality is true because of the positive definiteness of an inner product.
That means replacing $\mu_l$ with $(I - P)\mu_l$ does not increase \eqref{eq:J-0-formula-segment-l}, i.e., we can assume $\mu_l\in\operatorname{span}(g(\imdom))$ when searching for a minimizer.
Using this, we obtain analogously to the derivation of \eqref{eq:J-0-formula-segment-l}
\begin{align*}
    &\int_{\OO_l} \Vert g(x) - \mu_l \Vert_{\Sigma_l^{-\!1}\!\!, \eta} + \log \det \Sigma_l \dx\\
    =& \int_{\OO_l} \sqrt{\sum_{i=1}^{L} \frac{\left[V_l^T(g(x) - \mu_l)\right]_i^2}{\sigma_i} + \eta}\dx + \lebesgue{\OO_l} \sum_{i=1}^{L} \log \sigma_i\\
    =& \int_{\OO_l} \sqrt{\sum_{i=1}^{L-m} \frac{\left[v_i^T (g(x) - \mu_l)\right]^2}{\sigma_i} + \eta}\dx + \lebesgue{\OO_l} \sum_{i=1}^{L} \log \sigma_i.
\end{align*}
The last equality holds by orthogonality.
We observe that in this example the eigenvalues $\sigma_{L - m + 1}, \dots, \sigma_L$ are not present in the data term.
As a consequence, during minimization the functional benefits from bringing these values closer to $0$ since the data term does not change but $\log \det \Sigma_l$ decreases.
In this particular case, we even see that $J_0$ is unbounded from below because having the eigenvalues tending to $0$ will let the value of $J_0$ diverge to $-\infty$.
One could circumvent this divergence, e.g., by replacing $\log$ with $\log(\cdot + 1)$, which is bounded below for positive arguments resulting in a lower bound for $J_0$.
Nevertheless, the problem that $J_0$ benefits from bringing the eigenvalues closer to $0$ remains.
When applying $J_\varepsilon$ in its original version defined on $X_\varepsilon$ (cf. \eqref{eq:definition-J-epsilon}) to such an image $g$, the functional will push $\sigma_{L - m + 1}, \dots, \sigma_L$ to be equal to $\varepsilon^2$ in the minimum.
Since $\Pe$ is closed, such minimizers are still in $\Pe$.
The extension of $J_\varepsilon$ to the open set $\Pp$ (cf. \eqref{eq:extension-J-epsilon}) stayed coercive with respect to $\Sigma_l$ due to the addition of $I_{\Pe}$ since it prevents $J_\varepsilon$ from making the eigenvalues smaller than $\varepsilon^2$ and thus avoids the problem that $\Pp$ is open and the eigenvalues could approach $0$.
Note that here the term \emph{coercivity} has to be defined in a more general way than usual, describing that the functional shows a rapid growth at the extremes of the domain.
However, in the \gammalimit{} when $\varepsilon \to 0$, we lose the closedness of the admissible set without compensating for it.
Hence, $J_0$ does not satisfy a condition like coercivity (in the general sense explained above) in the eigenvalues of $\Sigma_l$ and we also lose the existence of a minimizer in general.%

\subsection{Bound for eigenvalues of covariance estimates}\label{subsec:epsAMS-boundedness-eigenvalues}
We have seen in \cref{subsec:epsAMS-unboundedness-J-0} that, for certain images $g$, we lose the boundedness from below when not regularizing with $\varepsilon$ and allowing $\Sigma_l \in \Pp$ for $l\in \indexset[k]$.
However, it was already hinted before that the data term $\Vert g(x) - \mu_l \Vert_{\Sigma_l^{-\!1}\!\!, \eta}$ and the $\log \det \Sigma_l$ term in \eqref{eq:def-J-0-formula} balance each other when the image $g$ is sufficiently benign and $\lebesgue{\OO_l} > 0$.
We first note that we can consider each summand of the outer sum over $l\in \indexset[k]$ in \eqref{eq:def-J-0-formula} separately and that we only have to consider the integral because only here the eigenvalues of $\Sigma_l$ play a role.
Furthermore, we will use that for a real-valued concave function $f$ with domain $D \subseteq \R$ the inequality
$    f\left(\sum_{i=1}^n a_i x_i\right) \geq \sum_{i=1}^{n} a_i f\left(x_i\right)$
holds for coefficients $a_1, \dots, a_n \geq 0$ with $\sum_{i=1}^{n} a_i = 1$ and $x_1, \dots, x_n \in D$.
Let $\Sigma_l = V_l D_l V_l^T$ be the eigendecomposition of $\Sigma_l$ with an orthogonal matrix $V_l \in \R^{L\times L}$ with columns $v_{l,i}$ for $i\in \indexset[L]$ and a diagonal matrix $D_l \in \R^{L \times L}$ with the eigenvalues of $\Sigma_l$ on its diagonal.
It follows with \eqref{eq:J-0-formula-segment-l}
\begin{align*}
    &\int_{\OO_l} \Vert g(x) - \mu_l \Vert_{\Sigma_l^{-\!1}\!\!, \eta} + \log \det \Sigma_l \dx\\
    \geq& \int_{\OO_l} \sqrt{\sum_{i=1}^{L} \frac{\left[V_l^T(g(x) - \mu_l)\right]_i^2}{\sigma(\Sigma_l)_i}}\dx + \lebesgue{\OO_l} \sum_{i=1}^{L} \log \sigma(\Sigma_l)_i\\
    =& \int_{\OO_l} \sqrt{L} \sqrt{\sum_{i=1}^{L} \frac{1}{L} \left[v_{l,i}^T\left(g(x) - \mu_l\right)\right]^2 \frac{1}{\sigma(\Sigma_l)_i}}\dx + \lebesgue{\OO_l} \sum_{i=1}^{L} \log \sigma(\Sigma_l)_i\\
    \geq& \int_{\OO_l} \frac{\sqrt{L}}{L} \sum_{i=1}^{L} \sqrt{\left[v_{l,i}^T\left(g(x) - \mu_l\right)\right]^2 \frac{1}{\sigma(\Sigma_l)_i}}\dx + \lebesgue{\OO_l} \sum_{i=1}^{L} \log \sigma(\Sigma_l)_i\\
    =& \int_{\OO_l} \frac{1}{\sqrt{L}} \sum_{i=1}^{L} \left\vert v_{l,i}^T\left(g(x) - \mu_l\right)\right\vert \frac{1}{\sqrt{\sigma(\Sigma_l)_i}}\dx + \lebesgue{\OO_l} \sum_{i=1}^{L} \log \sigma(\Sigma_l)_i\\
    =& \sum_{i=1}^{L} \left(\frac{1}{\sqrt{L}} \int_{\OO_l} \left\vert v_{l,i}^T\left(g(x) - \mu_l\right)\right\vert \dx \; \frac{1}{\sqrt{\sigma(\Sigma_l)_i}} + \lebesgue{\OO_l} \log \sigma(\Sigma_l)_i\right).
\end{align*}
The summands of the outer sum are uncoupled with respect to $\sigma(\Sigma_l)_i$, thus they can be considered separately.
For each $i\in \indexset[L]$, the summand is a function of $\sigma(\Sigma_l)_i$ of the form
\begin{equation}\label{eq:minimum-eigenvalues-function}
    a/\sqrt{x} + b \log x,
\end{equation}
for coefficients $a\geq 0$ and $b > 0$.
This function is defined for $x > 0$ and has a minimum at $x^* = \left(\frac{a}{2 b}\right)^2$ if $a > 0$.
For $x < x^*$, it is strictly monotonically decreasing and, for $x > x^*$, strictly monotonically increasing.
Hence, the minimum at $x^*$ is global.
Consequently, the following $\sigma(\Sigma_l)_i$ minimizes the $i$-th summand
\begin{equation}\label{eq:minimum-eigenvalue-covariance-J-0}
    \sigma(\Sigma_l)_i^* = \left(\int_{\OO_l} \left\vert v_{l,i}^T\left(g(x) - \mu_l\right)\right\vert \dx\Big/\left(2 \sqrt{L} \lebesgue{\OO_l}\right)\right)^2.
\end{equation}
While the denominator is always greater $0$ under the assumption of $\OO_l$ having positive measure, the integral in the numerator can be $0$.
This is the case when $v_{l,i}$ and $g(x) - \mu_l$ are perpendicular for almost every $x \in \OO_l$.
Geometrically speaking, we require that $g(x)$ deviates from $\mu_l$ in the direction of $v_{l,i}$ for almost every $x\in \OO_l$ or, in terms of spectra of images, that there is variation in the spectra of segment $l$ along the direction of $v_{l,i}$.
A visualization is given in \cref{fig:visualization-bound-eigenvalues}.
\begin{figure}[htp]
    \centering
    \pgfplotstableread{input/plot_feat_dist_seg_data.csv}{\seg}
\pgfplotstableread{input/plot_feat_dist_seg_mean_data.csv}{\segmean}

\begin{tikzpicture}[scale=0.99]
  \definecolor{color0}   {RGB}{  0  84 159}
  \definecolor{rwth-50}{RGB}{142 186 229}
  \definecolor{boundary_color0}{RGB}{0, 33, 65}

  \definecolor{color1}   {RGB}{204   7  30}
  \definecolor{rot-50}{RGB}{230 150 121}
  \definecolor{boundary_color1}{RGB}{64, 7, 30}

  \begin{axis}[
    xticklabel = \empty,
    yticklabel = \empty,
    width = .8\textwidth,
    height = 0.4\textwidth,
    legend cell align = {left},
    legend pos = north east,
    xlabel = {$g^1(x)^T v_{l,j}^1$},
    ylabel = {$g^1(x)^T v_{l,i}^1$},
  ]

  \addplot[draw=boundary_color0, fill=color0, only marks] table [x = {x}, y = {y}] {\seg};
  \addplot[draw=black, fill=rwth-50, only marks, mark size=5pt] table [x = {x}, y = {y}] {\segmean};

  \legend{
      Spectra from $g^1 \vert_{\OO_l}$,
      Segment's mean $\mu_l^1$
    }

    \node (source) at (axis cs:0.0, 0.0){};
    \node[label={[xshift=-5.5mm, yshift=-8mm, scale=1.5]{$v_{l,i}^1$}}] (destination) at (axis cs:0.0, 2500.0){};
    \draw[->, line width=1.mm](source)--(destination);

    \node (source) at (axis cs:0.0, 0.0){};
    \node[label={[xshift=-3mm, yshift=-1.5mm, scale=1.5]{$v_{l,j}^1$}}] (destination) at (axis cs:2500.0, 0.0){};
    \draw[->, line width=1.mm](source)--(destination);
\end{axis}

\end{tikzpicture}\\
    a) Samples from spectral distribution with variation in both depicted directions.\\[2ex]
    \pgfplotstableread{input/plot_feat_dist_seg_proj_data.csv}{\seg}
\pgfplotstableread{input/plot_feat_dist_seg_mean_proj_data.csv}{\segmean}

\begin{tikzpicture}[scale=0.99]
  \definecolor{color0}   {RGB}{  0  84 159}
  \definecolor{rwth-50}{RGB}{142 186 229}
  \definecolor{boundary_color0}{RGB}{0, 33, 65}

  \definecolor{color1}   {RGB}{204   7  30}
  \definecolor{rot-50}{RGB}{230 150 121}
  \definecolor{boundary_color1}{RGB}{64, 7, 30}

  \begin{axis}[
    xticklabel = \empty,
    ymin=-1.0, ymax=1.5,
    yticklabel = \empty,
    width = .8\textwidth,
    height = 0.4\textwidth,
    legend cell align = {left},
    legend pos = north east,
    xlabel = {$g^2(x)^T v_{l,j}^2$},
    ylabel = {$g^2(x)^T v_{l,i}^2$},
  ]

  \addplot[draw=boundary_color0, fill=color0, only marks] table [x = {x}, y = {y}] {\seg};
  \addplot[draw=black, fill=rwth-50, only marks, mark size=5pt] table [x = {x}, y = {y}] {\segmean};

  \legend{
        Spectra from $g^2\vert_{\OO_l}$,
        Segment's mean $\mu_l^2$
    }

    \node (source) at (axis cs:0.0, 0.0){};
    \node[label={[xshift=-6mm, yshift=-8mm, scale=1.5]{$v_{l,i}^2$}}] (destination) at (axis cs:0.0, 1.3){};
    \draw[->, line width=1.mm](source)--(destination);

    \node (source) at (axis cs:0.0, 0.0){};
    \node[label={[xshift=-6mm, yshift=-12mm, scale=1.5]{$v_{l,j}^2$}}] (destination) at (axis cs:2500.0, 0.0){};
    \draw[->, line width=1.mm](source)--(destination);
\end{axis}

\end{tikzpicture}\\
    b) Samples from spectral distribution with variation in only one direction.
    \caption{The two figures show samples from two exemplary spectral distributions of $g^m\vert_{\OO_l}$ for a segment $\OO_l\subseteq \imdom$ and images $g^m\in L^\infty(\imdom, \R^L)$, $m\in \{1,2\}$.
    The arrows $v_{l,i}^m \in \R^L$ and $v_{l,j}^m \in \R^L$ are the $i$-th and $j$-th eigenvector of the symmetric and positive-semidefinite covariance matrix $\Sigma_l^m \in \R^{L\times L}$.
    The projections of the sampled spectra onto those vectors are plotted.
    In the top figure, there is variation of the spectra in both directions $v_{l,i}^1$ and $v_{l,j}^1$, which results in estimates for the $i$-th and $j$-th eigenvalue of $\Sigma_l^1$ being bounded from below with a positive bound (cf. \cref{eq:minimum-eigenvalue-covariance-J-0}) and hence $\Sigma_l^1 \in \Pp$ (provided that we also have variation in the directions of all other eigenvectors $v_{l,t}^1$ for $t\in \indexset[L]\setminus \{i,j\}$).
    In the bottom figure, there is only variation of the spectra in direction $v_{l,j}^2$, not in direction $v_{l,i}^2$.
    This results in an estimate for the $i$-th eigenvalue of $\Sigma_l^2$ equal to $0$, which is not admissible anymore since then $\Sigma_l^2 \notin \Pp$.
    Technically, we have $a = 0$ in \eqref{eq:minimum-eigenvalues-function} in this situation, making the function unbounded from below and \cref{eq:minimum-eigenvalue-covariance-J-0} infeasible.}
    \label{fig:visualization-bound-eigenvalues}
\end{figure}
A heuristic for practical applications to achieve that $\sigma(\Sigma_l)_i^* > 0$ for all $i\in \indexset[L]$ is to apply the \ac{pca} \cite{Pe01} to the data before processing it and keep only the components needed to describe 99.9\% of the variance.

\section{Conclusion}\label{sec:conclusion}
We analyzed the distribution-dependent \acl{ms} model for hyperspectral image segmentation and its covariance regularization with a parameter $\varepsilon$.
We showed the existence of minimizers and the \gammaconvergence{} of the functional for $\varepsilon \to 0$.
Moreover, we gave an example of an image where the \gammalimit{} does not have a minimizer.
It turned out that the regularization with $\varepsilon$ plays a very important role in the segmentation model since it ensures the feasibility of the model, allows to find bounds for the minimizing sequence such that convergent subsequences can be chosen and most importantly makes the set of admissible covariance matrices closed.
In the \gammalimit{}, the set of admissible covariance matrices is not closed anymore.
Together with the expression for the minimum values of the eigenvalues that was derived, we saw that this allows to push the eigenvalues of the covariance matrices arbitrarily close to $0$ and therefore the covariance matrices out of the admissible set when there is no variation in the data in direction of the corresponding eigenvectors.
Consequently, the functional in its current form needs the $\varepsilon$-regularization.

We plan to follow two future directions: we will investigate conditions on the data that guarantee the existence of minimizers %
in the \gammalimit{}, and without altering the functional.
The formula for the minimum eigenvalues can serve as good starting point since we need the numerator to be always greater than a positive constant.
Moreover, as already mentioned, the set of admissible covariance matrices is not closed anymore in the \gammalimit{}.
Therefore, we will secondly aim to find a modification of the functional that guarantees the existence of minimizers without regularizing the covariance estimates or any condition on the data.
An idea how to achieve that is to alter the model such that it fulfills a coercivity condition that makes the functional values explode when an eigenvalue of a covariance matrix is approaching $0$ and running out of the admissible set.
The point at which the functional explodes, when the eigenvalues fall below it, should, however, depend on the data.
Also for this future direction the expression for the minimum eigenvalues is a good starting point.

\begin{acknowledgement}
    Jan-Christopher Cohrs was funded by the Deutsche Forschungsgemeinschaft (DFG, German Research Foundation) - 333849990/GRK2379 (IRTG Modern Inverse Problems).
\end{acknowledgement}
\ethics{Competing Interests}{
    The authors have no conflicts of interest to declare that are relevant to the content of this chapter.
}

\bibliographystyle{spmpsci}
\bibliography{references}

\end{document}